\documentclass{amsart}

\usepackage{bbm}
\usepackage{color}
\numberwithin{equation}{section} 
\numberwithin{figure}{section} 
\theoremstyle{theorem}
\newtheorem{thm}{Theorem}
\newtheorem{cor}{Corollary}[section]
\newtheorem{lem}[cor]{Lemma}
 \newtheorem{prop}[cor]{Proposition}
 \theoremstyle{remark}
  \newtheorem{rem}{Remark}

\newcommand{\1}{\mathbbm{1}}
\newcommand{\e}{\mathrm e}
\DeclareMathOperator{\Int}{Int}
\renewcommand{\phi}{\varphi}

\renewcommand{\tilde}{\widetilde}

\def\R{\mathbb{R}}

\def\N{\mathbb{N}}

\def\P{{\mathcal P}}

\def\R{\mathbb{R}}
\def\X{\mathbb{H}}

\def\N{\mathbb{N}}

\def\P{{\mathbb P}}

\def\ddelta{\boldsymbol{\delta}}

\def\h{\Phi}

\begin{document}


\title[A Fr\'{e}chet law and an Erd{\H o}s-Philipp law ]{A Fr\'{e}chet law and an Erd{\H o}s-Philipp law for maximal cuspidal windings}

\author{JOHANNES JAERISCH, MARC KESSEB{\"O}HMER \\ and  BERND O. STRATMANN}

\address{Department of Mathematics, Graduate School of Science\\ Osaka University, 1-1 Machikaneyama\\
Toyonaka, Osaka, 560-0043 Japan}
\email{jaerisch@cr.math.sci.osaka-u.ac.jp}

\address{Fachbereich 3 -- Mathematik und Informatik, Universit\"at Bremen \\ Bibliothekstra\ss e
1, 28359 Bremen, Germany.}
\email{mhk@math.uni-bremen.de, bos@math.uni-bremen.de}


\begin{abstract}
In this paper we establish a Fr\'{e}chet law for maximal cuspidal
windings of the geodesic flow on a Riemannian surface associated with
an arbitrary finitely generated, essentially free Fuchsian group with
parabolic elements. This result extends previous work by Galambos and Dolgopyat
and is obtained by applying Extreme Value Theory. Subsequently, we show
that this law gives rise to an Erd{\H o}s-Philipp law and to various generalised Khintchine-type
results for maximal cuspidal windings. These results strengthen previous results by Sullivan, Stratmann and Velani for Kleinian groups, and extend earlier
work by  Philipp on continued
fractions, which was inspired by a conjecture of Erd{\H o}s.
\end{abstract}

\maketitle

\section{Introduction and Statements of Result}
 
We establish a Fr\'{e}chet law and an Erd{\H o}s-Philipp law for maximal cuspidal
windings of the geodesic flow on $\X/G$,  for a finitely generated, essentially
free Fuchsian group $G$ acting on the upper half-space model $(\X,d)$
of $2$-dimensional hyperbolic space.  Recall that to each $\xi$ in the radial limit set $L_{r}(G)$ of $G$ one can associate an infinite
word expansion in the symmetric set $G_{0}$ of generators of $G$. Namely, with $F$ referring to the Dirichlet fundamental domain
of $G$ at $i \in \X$,  the images of $F$ under $G$ tesselate $\X$ and each side of each of the tiles is uniquely labeled
by an element of $G_{0}$.  The hyperbolic ray $s_{\xi}$ from $i$ towards $\xi\in L_{r}(G)$ has to traverse infinitely many of these tiles, and  the infinite
word expansion associated with $\xi$ is then obtained by progressively recording, starting at $i$,  the generators of the sides at which $s_{\xi}$ exits the tiles. In this way we derive an infinite reduced word on the alphabet $G_{0}$. We then form blocks as follows. Each hyperbolic generator in this word has block length $1$. Further, if there is a block in which the same parabolic generator appears $n$ times and
if there is no larger block of this parabolic generator containing that block, then this block is of length $n$. This allows us to define the process $(X_{k})$
by setting $X_{k}$ to be equal to the length of the $k$-th block. By construction, such a block of length $n$ corresponds to
the event that the projection of $s_{\xi}$ onto $\X/G$ spirals precisely $n-1$ times around a cusp of $\X/G$.
The main results of this paper will establish  asymptotic estimates and strong distributional convergence for
 the process  $(Y_{n})$, given by
\[
Y_{n}:=\max_{k=1,\dots,n}X_{k}.
\]

\begin{thm}[Fr\'{e}chet law for maximal cuspidal windings]\label{thm} For
each essentially free, finitely generated Fuchsian group $G$ with
parabolic elements and with exponent of convergence $\delta=\delta(G)$,  the following holds. For every $s>0$  and for each
probability measure $\nu$ absolutely continuous with respect to the
Patterson measure $m_{\delta}$ of $G$, we have that
\[
\lim_{n\to\infty}\nu\left(\left\{Y_{n}^{2\delta-1} / n\le s\right\} \right)=\exp \left(- \kappa\left(G\right)/s\right).
\]
Here, the constant $\kappa\left(G\right)$ is explicitly given by
\begin{align*}
\kappa(G): & =\sum_{\gamma\in\Gamma_{0}}\left(\h\left(p_{\gamma}\right)w_{\gamma}^{-\delta}\right)^{2} /  (\left(2\delta-1\right) \, \mu_{\delta}\left(\left\{ X_{1}=1\right\} \right)),
\end{align*}
where $\Gamma_{0}$ refers to the symmetric set of parabolic generators of $G$, $p_{\gamma}$ to the parabolic fixed point and $w_{\gamma}$ to the width of the cusp associated with
 $\gamma\in\Gamma_{0}$ (see the definition prior to Lemma \ref{prop:asymptotic-n-excursion}), $\mu_{\delta}$ to the (up to a multiplicative constant) unique measure absolutely
continuous with respect to $m_{\delta}$ which is invariant under
the Bowen-Series map (see Section 2.1 and  Section 2.2), and $\Phi$
to a version of the Radon-Nikodym derivative $d\mu_{\delta}/dm_{\delta}$
(see (\ref{eq:eigenfunction-of-PF}) for the definition).

\end{thm}

Let us remark that the above constant $\kappa(G)$ is strictly positive, since by a result of Beardon \cite{Beardon2} we have that if $G$ has
parabolic elements, then $\delta(G) > 1/2$.
Also, note that Theorem \ref{thm} gives an answer to a question asked by Pollicott
in \cite{MR1500147}, where he shows that a result by Galambos
\cite{MR0299576,MR0352028} can be rephrased in terms of the modular
group. The reader might like to recall that this result of Galambos
states that for all $s>0$ and for each probability measure $\nu$
absolutely continuous with respect to the Lebesgue measure on $\left(0,1\right)$,
we have that
$
\lim_{n\rightarrow\infty}\nu\left(\left\{ (\max_{k=1,\dots,n}a_{k})/n\le s\right\} \right)=\exp\left(-1/(s\log2)\right).
$
Here,  $a_{k}(x)$
refers to the $k$-th entry in the regular continued fraction
expansion of $x\in\left(0,1\right)$. Let us remark that a straightforward
adaptation of our proof of  Theorem \ref{thm} gives an alternative proof
of this result of Galambos. Moreover, note that Dolgopyat  studied statistical properties of the geodesic flow on negatively curved cusped surfaces 
of finite area. He shows that if these surfaces have constant negative curvature, then the maximal cuspidal excursions  satisfy a  Fr\'echet law 
with respect to Lebesgue  measure (see \cite[Section 4.2, Corollary 3]{Dol}). 

We will also give some interesting applications of the result in
Theorem \ref{thm}. These include the following theorem, whose  first
assertion extends a result of Philipp in \cite[Theorem 1]{MR0387226}, who showed that for Lebesgue almost every $x \in (0,1)$ we have that $\liminf_{n\rightarrow\infty}\,  \max_{k=1,...,n} a_{k}(x) (\log\log n)/n= 1/\log 2$.
This settled a conjecture by Erd{\H o}s  (see \cite{MR0387226}), who had previously conjectured that  the above {\em Limes inferior} is equal to $1$.  Also, note that
the second assertion of  {Theorem \ref{prop:loglogliminf} } extends  \cite[Corollary  to Theorem 3]{MR0387226}.

\begin{thm}[Erd{\H o}s-Philipp law for maximal cuspidal windings]
\label{prop:loglogliminf}For  $G$ as above, we have $m_{\delta}$
almost everywhere that
\[
\liminf_{n\rightarrow\infty}\, Y_{n}^{2\delta-1} (\log\log n)/n=\kappa\left(G\right).
\]
Moreover, for each  sequence $\left(\ell_{n}\right)$ of
positive reals we have $m_{\delta}$
almost everywhere that
\[
\limsup_{n\rightarrow\infty}Y_{n}/\ell_{n}\in\left\{ 0,\infty\right\} .
\]

\end{thm}
Theorem \ref{prop:loglogliminf} has some further interesting consequences. Namely, it
permits the derivation of the following Khintchine-type results, where $\xi_{t}$
denotes the unique point on the hyperbolic ray from $i\in\X$ towards
an element $\xi\in L_{r}\left(G\right)$ such that $d\left(i,\xi_{t}\right)=t$,
for $t\ge0$.
\begin{cor}
\label{cor:log-khintchine}For $G$ as above, we have for $m_{\delta}$
almost every $\xi\in L_{r}\left(G\right)$ that

\[
\lim_{n\rightarrow\infty}\frac{\log Y_{n}\left(\xi\right)}{\log n}=\frac{1}{2\delta-1}\,\mbox{ and }\,\lim_{T\rightarrow\infty}\max_{0\le t\le T}\frac{d\left(\xi_{t},G\left(i\right)\right)}{\log T}=\frac{1}{2\delta-1}.
\]

\end{cor}
Note that  in here the
second assertion  represents a significant strengthening  of the result  that
$\limsup_{t\rightarrow\infty}d\left(\xi_{t},G\left(i\right)\right) /\log t=(2\delta-1)^{-1}$ for $m_{\delta}$ almost everywhere $\xi$, which  was obtained in \cite{MR1327939}
for arbitrary geometrically finite Kleinian groups with parabolic
elements, generalising work of Sullivan for cofinite Kleinian
groups \cite{MR688349}. Let us point out that in this result  the {\em Limes superior} cannot be replaced
by a  {\em Limes inferior}. Also,  let us remark that the first statement in the above corollary is closely related to the well-known result by Khintchine
 for continued
fractions (\cite{khinchin-continuedfractionsMR0161833}), which asserts that Lebesgue almost everywhere, we have that $\limsup_{n \to \infty}\log a_{n}/\log n=1$.
In fact, by using the Limsup-Max Principle stated in (\ref{eq:maximum-exercise}), it immediately follows from Corollary \ref{cor:log-khintchine}  that for   essentially free, finitely generated Fuchsian groups with parabolic elements we have $m_{\delta}$ almost everywhere that
 $\limsup_{n\rightarrow\infty} \log X_{n}/ \log n=(2\delta-1)^{-1}$.
Again, let us remark that in here  the {\em Limes superior} can not be replaced
by a  {\em Limes inferior}.

 Further, we point out that the above mentioned results by Galambos, Philipp and  Dolgopyat exclusively
  concern dynamical systems for which the limit set is either the unit interval or the whole boundary of hyperbolic space, and in these situations the relevant measure 
  is  equal to the Lebesgue measure. In contrast to this, the conformal dynamical systems which we consider have limit sets which are of fractal nature.
Hence,  one of the novelties of our analysis is that we obtain strong distributional convergence and asymptotic estimates
for processes which are defined on conformal attractors with parabolic elements and of Hausdorff dimension strictly  less than $1$.

Finally, let us also remark that our results have their natural home in
 Extreme Value Theory, which  studies distributional properties
of the sequence of maxima  $\max\left(Z_{1},\dots,Z_{n}\right)$  for
some stationary process $(Z_{k})$. For independent
and identically distributed processes, classical Extreme Value Theory
asserts that there are only three types of non-trivial limiting distributions possible. Namely,  one of the
following three distributions has to occur, where $\alpha>0$ denotes the shape parameter and $\beta>0$ the scale parameter:
\begin{itemize}
\item[] {\it Gumbel extreme value distribution:} \, \, \, \, $ \e^{-\e^{-y/\beta}}$ \,\,\,\,\,\, \,\, for \, \, \, $y\in\R$.
\item[] {\it Fr\'{e}chet extreme value distribution:} \, \,
$
\left\{ \begin{array}{ccc}
0 & \text{    \;    for } &y\leq0\\
 \e^{-(\beta/y)^{\alpha}} & \text{   \,\,     for } &y>0
. \end{array}\right.$
\item[] {\it Weibull extreme value distribution:} \, \,
$
\left\{ \begin{array}{ccc}
\e^{-(-y/\beta)^{ \alpha}} & \text{ for } &  y\leq0\\
1 & \text{ for } &y>0. \end{array}\right.$
\end{itemize}
This result goes back to Fisher and Tippett (\cite{54.0560.05}) and it
was independently proven and extended by Gnedenko (\cite{MR0008655},
see also \cite{MR0096342}). Also, the Fr\'{e}chet distribution was
first obtained in \cite{54.0562.06}.  Conditions under which the same
classification of limiting distributions holds for dependent random
variables have been considered by various authors (see for instance
\cite{MR0065122,MR0161365,MR0176530,MR0362465,MR528323}).
Very recently, Extreme Value Theory has also been under investigation in the context
of dynamical systems (see for instance \cite{MR1340561,MR1827111,MR2422964,MR2437222,MR2639719,MR2643710,MR2749711,MR2669635,GuptaHollanderNicol2011}).
In \cite{MR1973926}  error bounds on
the distributional convergence and certain refinements are investigated
in the general framework of parabolic Markov fibered systems.  In these studies, the authors
also consider the Jacobi-Perron algorithm and derive asymptotic estimates for
its rare events with respect to the underlying invariant measure.  These results are very much in the spirit of the outcomes of Theorem \ref{prop:loglogliminf}. Therefore, in the terminology of  Extreme Value Theory, we can  now say that Theorem \ref{thm}  in particular shows that the extreme value distribution of the process $(Y_{n}^{2\delta-1}/n)$ is
Fr\'{e}chet with shape parameter equal to $1$ and  scale parameter equal to $\kappa(G)$. Note that this result complements a recent result of Ferguson
\cite[Theorem 4.5]{thesis}, who studied topologically mixing $C^{1+\alpha}$-maps
on a compact Riemannian manifold and showed that the extreme value distribution of the process  obtained from the closest visit of orbits to some arbitrarily
chosen point is Gumble with respect to the invariant probability measure absolutely continuous with respect to the Lebesgue measure.

The paper is organised as follows. In Section 2 we collect the necessary
preliminaries of Ergodic Theory for Fuchsian groups and then give
a brief discussion of a result from Extreme Value Theory which is relevant for this
paper. Finally, the proofs of Theorem \ref{thm}, Theorem
\ref{prop:loglogliminf} and Corollary \ref{cor:log-khintchine} are given in Section 3.

\begin{rem}  It seems very likely that one can adapt   
 the methods of this paper in order to establish similar 
results in more general settings. For instance, this should be possible for more general geometrically finite Kleinian groups with parabolic elements of different ranks as well as  for Hadamard manifolds with cusps, possibly under some mild assumptions on the
growth rate of their associated maximal parabolic subgroups. However, for ease of exposition and to guarantee that the methods we use remain sufficiently transparent,  we restrict the discussion in this paper to 
essentially free, finitely generated Fuchsian group with
parabolic elements.
\end{rem}

\section{Preliminaries}
In this section we discuss some results from Ergodic Theory for
Fuchsian groups with parabolic elements and from Extreme Value Theory.
In Section 2.1 we give a brief account of the way the action of $G$ on the boundary
$\partial\X$ of $\X$ can be represented by a canonical Markov map,
often  referred to as the Bowen-Series map. In Section 2.2 we recall
some basics from Patterson measure theory. This will  include a discussion of  the measure which is invariant under
the Bowen-Series map and which is absolutely continuous with respect to the Patterson measure. Here, we also derive two lemmata
which will be required in the proofs of the main results. In
Section 2.3 we summarise  strong mixing properties for
the induced dynamics of the Bowen-Series map on the complement of
some neighbourhood of the parabolic fixed points of $G$. Further, we give two lemmata which
will turn out to be useful in Section 3. Finally,
in Section 2.4, we recall a result from Extreme Value Theory obtained in
\cite{MR691492}. This result will be crucial in the proof of
Theorem \ref{thm}.

\subsection{The Canonical Markov Map}
As already mentioned in the introduction, throughout this paper we exclusively
consider a finitely generated, essentially free Fuchsian group $G$ with parabolic elements.
By definition (see \cite{MR2041265}), a group of this type can be
written as a free product $G=H\ast\Gamma$, where $H=\langle h_{1},h_{1}^{-1}\rangle\ast\ldots\ast\langle h_{u},h_{u}^{-1}\rangle$
denotes the free product of finitely many elementary hyperbolic subgroups of $G$,
and $\Gamma=\langle\gamma_{1},\gamma_{1}^{-1}\rangle\ast\ldots\ast \langle\gamma_{v},\gamma_{v}^{-1}\rangle$ denotes the free
product of finitely many elementary parabolic subgroups of $G$. Also, we let $H_{0}:= \{h_{1},h_{1}^{-1}, ..., h_{u},h_{u}^{-1}\}$ and $\Gamma_{0}:=
\{\gamma_{1},\gamma_{1}^{-1}, ... ,\gamma_{v},\gamma_{v}^{-1}\} $, so that $G_{0}=H_{0} \cup \Gamma_{0}$ is the symmetric set of all generators of $G$. Note that for each $\gamma \in \Gamma_{0}$ we have that  $\langle \gamma,\gamma^{-1}\rangle\cong{\Bbb Z}$ and that
 there exists a unique fixed point $p_{\gamma} \in \partial \X$ such that
$\gamma^{\pm 1}(p_{\gamma})=p_{\gamma}$. Without loss of generality, we  always
assume that $ \{\infty\} \in \partial \X$ is not contained in the limit set $L(G)$ of $G$. Also, note that, by construction, these Fuchsian
groups do not have any relations. Moreover, we can
assume that $G$ admits the choice of a Dirichlet fundamental domain $F$ at $i$, which is bounded within $\X$ by
a finite set $\mathcal{F}$ of sides. Let us now first recall from
\cite{MR2158407} the construction of the relevant coding map $T$
associated with $G$, which maps the radial limit set $L_{r}(G)$
into itself. This construction parallels the construction of the well-known
Bowen-Series map (see \cite{MR556585}, \cite{MR2098779}). For $\xi,\eta\in L_{r}(G)$,
let $\ell_{\xi,\eta}:\R\to\X$ denote the directed geodesic from $\eta$
to $\xi$ such that $\ell_{\xi,\eta}$ intersects the closure $\overline{F}$
of $F$ in $\X$, and normalised such that $\ell_{\xi,\eta}(0)$ is
the summit of $\ell_{\xi,\eta}$. We then define the exit time $e_{\xi,\eta}$  by
\[
e_{\xi,\eta}:=\sup\left\{ s:\ell_{\xi,\eta}(s)\in\overline{F}\right\} .
\]
 Since $\xi,\eta\in L_{r}(G)$, we clearly have that $|e_{\xi,\eta}|<\infty$.
By Poincar\'e's polyhedron theorem (see \cite{MR1279064}), we have
that the set $\mathcal{F}$ carries an involution $\iota:\mathcal{F}\to\mathcal{F}$
such that for each $s\in\mathcal{F}$
there is a unique {side-pairing} transformation $g_{s}\in G_{0}$ for which $g_{s}(\overline{F})\cap\overline{F}=\iota(s)$. We then let
\[
\mathcal{L}_{r}(G):=\{(\xi,\eta):\xi,\eta\in L_{r}(G),\xi\neq\eta\mbox{ and there exists } t\in\R \mbox{ such that }\ell_{\xi,\eta}(t)\in\overline{F}\},
\]
 and define the map $S:\mathcal{L}_{r}(G)\to\mathcal{L}_{r}(G)$,
for all $(\xi,\eta)\in\mathcal{L}_{r}(G)$ such that $\ell_{\xi,\eta}(e_{\xi,\eta})\in s$,
for some $s\in\mathcal{F}$, by
\[
S(\xi,\eta):=(g_{s}(\xi),g_{s}(\eta)).
\]
 In order to show that the map $S$ admits a Markov partition, we
introduce the following collection of subsets of the boundary $\partial\X$
of $\X$. For $s\in\mathcal{F}$, let $A_{g_{s}}$ refer to the open hyperbolic
half-space for which $F\subset\X\setminus A_{g_{s}}$ and $s\subset\partial A_{g_{s}}$.
Moreover,  let $\Pi:\X\to\partial\X$ denote the shadow-projection
given by $\Pi(E):=\{\xi\in\partial\X:s_{\xi}\cap E\neq\emptyset\}$, for $E \subset \X$ and  where $s_{\xi}$ denotes the hyperbolic ray from $i$ towards $\xi\in \partial\X$.
We then define the projection $a_{g_{s}}$ of the side $s$ to $\partial\X$
by
\[
a_{g_{s}}:=\mbox{Int}\left(\Pi(A_{g_{s}})\right).
\]
 Clearly, we have that $a_{g_{s}}\cap a_{g_{t}}=\emptyset$, for all $s,t\in\mathcal{F},s\neq t$.
Hence, by convexity of $F$, we have that $\ell_{\xi,\eta}(e_{\xi,\eta})\in s$
if and only if $\xi\in a_{g_{s}}$. In other words, $S(\xi,\eta)=(g_{s}\xi,g_{s}\eta)$
for all $\xi\in a_{g_{s}}$. This immediately gives that the projection
map $\pi:(\xi,\eta)\mapsto\xi$ onto the first coordinate of $\mathcal{L}_{r}(G)$
leads to a canonical factor $T$ of $S$, that is, we obtain the map
\[
T:L_{r}(G)\to L_{r}(G),\mbox{ given by }T\arrowvert_{a_{g_{s}}\cap L_{r}(G)} :=g_{s}.
\]
 Obviously, $T$ satisfies $\pi\circ S=T\circ\pi$. Since
  $T(a_{g_{s}})=g_{s}(a_{g_{s}})=\mbox{Int}(\partial\X\setminus a_{\iota(s)})$,
it follows that $T$ is a non-invertible Markov map with respect to
the partition $\{a_{g_{s}}\cap L_{r}(G):s\in\mathcal{F}\}$. For this
so-obtained coding map $T$ we then have the following result.
\begin{prop}[{{\cite[Proposition 2, Proposition 3]{MR2158407}}}] The map $T$
is a topologically mixing Markov map with respect to the partition
generated by $\alpha:=\{a_{g}\cap L_{r}(G): g \in G_{0}\}$. Moreover,
the map $S$ is the natural extension of $T$.
\end{prop}

\subsection{Patterson Measure Theory}

\noindent In order to introduce the $T$-invariant measure $\mu_{\delta}$ supported on the limit set $L(G)$ of $G$ which will be relevant for our purposes, we first briefly recall some of the highlights
in connection with the Patterson measure and the Patterson-Sullivan
measure (for more detailed discussions of these measures we refer to \cite{MR0450547,MR556586,MR634434,MR1041575,MR1327939}).
By now it is folklore that, for a fixed given sequence $(s_{n})$ of positive
reals which tends to $\delta$ from above,  a Patterson measure is given by some weak limit of the sequence of probability measures

\[
\left(\left(\sum_{g\in G}\exp(-s_{n}d(i,g(i)))\right)^{-1}\sum_{g\in G} \exp\left(-s_{n}d(i,g(i)))\ddelta_{g(i)}\right)\right),
\]
where $\ddelta_{x}$ refers to the Dirac point measure supported on  $x$.
Clearly, such a Patterson measure is always
 a probability measure which is supported on $L(G)$.
For finitely generated, essentially free Fuchsian groups it is well known
that a so-obtained limit measure is unique and non-atomic, and does not depend
on the particular choice of the sequence $(s_{n})$.  In the sequel, this unique Patterson measure will always be denoted by $m_{\delta}$.  It is well-known that $m_{\delta}$ has the property of being $\delta$-conformal, that is, for all $g\in G$
and $\xi\in L(G)$, we have that
\begin{equation}
\frac{d\,(m_{\delta}\circ g)}{d\, m_{\delta}}(\xi)=\left|g'\left(\xi\right)\right|^{\delta}.\label{eq:delta-conformality}
\end{equation}
 The $\delta$-conformality of  $m_{\delta}$ is one of the key properties of the Patterson measure
and  we will now briefly recall a
very convenient geometrisation of this property. For this, fix some $r_{0}>0$ and let $B(\xi_{t})\subset\X$
denote the hyperbolic disc centered at $\xi_{t}$ of hyperbolic radius
$r_{0}$. We then have the following generalisation of Sullivan's shadow lemma.
\begin{prop}[{{\cite{MR766265,MR1327939}}}] \label{prop:measure}  There exists a constant $C>0$ such that for each $\xi\in L(G)$ and $t\geq0$,
we have that
\[
C^{-1} \exp \left((1-\delta)d(\xi_{t},G(i)) -\delta t \right) \leq m_{\delta}(\Pi(B(\xi_{t}))) \leq  C \exp\left((1-\delta)d(\xi_{t},G(i)) -\delta t \right).
\]

\end{prop}
\noindent We continue our quick journey through Patterson measure theory by noting that $m_{\delta} $ gives rise to a measure $\tilde{m}_{\delta}$
on $(L(G)\times L(G))\setminus\{\mbox{diagonal}\}$, which is ergodic
with respect to the action of $G$ on $(L(G)\times L(G))\setminus\{\mbox{diagonal}\}$, where we recall that this action is
given by $g((\xi,\eta))=(g(\xi),g(\eta))$. The measure $\tilde{m}_{\delta}$ is usually
referred to as the Patterson-Sullivan measure and it is defined by
\[
d\tilde{m}_{\delta}(\xi,\eta):=\frac{dm_{\delta}(\xi)dm_{\delta}(\eta)}{|\xi-\eta|^{2\delta}}.
\]
The (first) marginal measure of the Patterson-Sullivan measure, restricted
to the set
\[
\mathcal{L}(G):=\{(\xi,\eta):\xi,\eta\in L(G),\xi\neq\eta\mbox{ and there exists } t\in\R\ \mbox{ such that } \ell_{\xi,\eta}(t)\in\overline{F}\},
\]
then defines the measure $\mu_{\delta}$ on $L(G)$, which is given by
\[
\mu_{\delta}:=\tilde{m}_{\delta}\big|_{\mathcal{L}\left(G\right)}\circ\pi^{-1}.
\]
For the system $\left(L_{r}(G),T,\mu_{\delta}\right)$
the following result has been obtained in \cite{MR2158407}.
\begin{prop}[{{\cite{MR2158407}}}]
\label{prop2} The map $T$ is measure preserving, conservative and
ergodic with respect to the infinite, $\sigma$-finite measure $\mu_{\delta}$.
\end{prop}

We end this section by giving two slightly technical observations which will be required in the proofs of
our main results.
In order to state these, we define the function $\h:L(G) \to \R$ by
\begin{equation}
\h\left(\xi\right):=\int\1_{\mathcal{L}(G)}\left(\xi,\eta\right)\left|\xi-\eta\right|^{-2\delta}dm_{\delta}(\eta).\label{eq:eigenfunction-of-PF}
\end{equation}
Note that $m_{\delta}$ almost everywhere we have that  $\h=d\mu_{\delta}/dm_{\delta}$.
Furthermore, note that for each $g\in G_{0}$ and  $\xi\in a_{g}\setminus\left\{ p_{\gamma}:\gamma\in\Gamma_{0}\right\} $, we have that
\[
\h\left(\xi\right)=\int_{\partial\X\setminus a_{g}}\left|\xi-\eta\right|^{-2\delta}dm_{\delta}(\eta),
\]
whereas, for $\gamma\in\Gamma_{0}$ one immediately verifies that
\[
\h\left(p_{\gamma}\right)=\int_{\partial\X\setminus(a_{\gamma}\,\cup\, a_{\gamma^{-1}})}\left|p_{\gamma}-\eta\right|^{-2\delta}dm_{\delta}(\eta).
\]
Let us also remark that for  $g\in H_{0}$ and $\xi\in a_{g}$, we obviously have that
\[
\lim_{\omega\rightarrow\xi}\sup_{\eta\in\partial\X\setminus a_{g}}\left\{ \left|\left|\omega-\eta\right|^{-2\delta}-\left|\xi-\eta\right|^{-2\delta}\right|\right\} =0.
\]
Using this, one then immediately sees,  that in this case the restriction $\h\big|_{a_{g}}$ of $\h$ to $a_{g}$ is
a continuous function which is bounded away
from zero and infinity.

\vspace{2mm}

Let us now come to the first of our technical observations. This observation will be required for the explicit computation of the constant $\kappa(G)$
in Theorem \ref{thm}.
\begin{lem}
\label{lem:fixpoint-of-perron-frobenius}For each $\gamma\in\Gamma_{0}$,
we have that
\[
\sum_{g\in G_{0}\setminus\left\{ \gamma^{\pm1}\right\} }\left|g'\left(p_{\gamma}\right)\right|^{\delta}\h\left(g\left(p_{\gamma}\right)\right)=\h\left(p_{\gamma}\right).
\]
\end{lem}
\begin{proof}
First note that, by definition of $\h$,
 we have for each $\gamma\in\Gamma_{0}$  that
\[
\sum_{g\in G_{0}\setminus\left\{ \gamma^{\pm1}\right\} }\left|g'\left(p_{\gamma}\right)\right|^{\delta}\h\left(g\left(p_{\gamma}\right)\right)=\sum_{g\in G_{0}\setminus\left\{ \gamma^{\pm1}\right\} }\left|g'\left(p_{\gamma}\right)\right|^{\delta}\int_{\partial\X\setminus a_{g}}\left|g\left(p_{\gamma}\right)-\eta\right|^{-2\delta}dm_{\delta}\left(\eta\right).
\]
Also, recall that $\left|g\left(\xi_{1}\right)-g\left(\xi_{2}\right)\right|^{2}=\left|g'\left(\xi_{1}\right)\right|\left|g'\left(\xi_{2}\right)\right|\left|\xi_{1}-\xi_{2}\right|^{2}$, for each M\"obius transformation $g$ and for all $\xi_{1},\xi_{2}\in\R$ (see for instance \cite{Beardon}).
By setting  $\xi_{1}:=p_{\gamma}$ and $\xi_{2}:=g^{-1}\left(\eta\right)$ and using the fact that by the chain rule we have
that $g'\left(g^{-1}\left(\eta\right)\right)\left(g^{-1}\right)'\left(\eta\right)=1$,
it follows that
\[
\left|g\left(p_{\gamma}\right)-\eta\right|^{2}=\left|g'\left(p_{\gamma}\right)\right|\left|\left(g^{-1}\right)'\left(\eta\right)\right|^{-1}\left|p_{\gamma}-g^{-1}\left(\eta\right)\right|^{2}.
\]
By combining these two observations and using the $\delta$-conformality of $m_{\delta}$, we obtain that
\[  \sum_{g\in G_{0}\setminus\left\{ \gamma^{\pm1}\right\} }\left|g'\left(p_{\gamma}\right)\right|^{\delta}\h\left(g\left(p_{\gamma}\right)\right)
  =  \sum_{g\in G_{0}\setminus\left\{ \gamma^{\pm1}\right\} }\int_{\partial\X\setminus a_{g}}\frac{\left|\left(g^{-1}\right)'\left(\eta\right)\right|^{\delta}}{\left|p_{\gamma}-g^{-1}\left(\eta\right)\right|^{2\delta}}dm_{\delta}\left(\eta\right).
\]
\begin{eqnarray*}
 & = & \sum_{g\in G_{0}\setminus\left\{ \gamma^{\pm1}\right\} }\int_{\partial\X\setminus a_{g}}\left|\left(g^{-1}\right)'\left(\eta\right)\right|^{\delta}\left|p_{\gamma}-g^{-1}\left(\eta\right)\right|^{-2\delta}dm_{\delta}\left(\eta\right)\\
 & = & \sum_{g\in G_{0}\setminus\left\{ \gamma^{\pm1}\right\} }\int_{\partial\X\setminus a_{g}}\left|p_{\gamma}-g^{-1}\left(\eta\right)\right|^{-2\delta}d\left(m_{\delta}\circ g^{-1}\right)\left(\eta\right)\\
 & = & \sum_{g\in G_{0}\setminus\left\{ \gamma^{\pm1}\right\} }\int\1_{\partial\X\setminus a_{g}}\left(g\left(\eta\right)\right)\left|p_{\gamma}-\eta\right|^{-2\delta}dm_{\delta}\left(\eta\right)\\
 & = & \sum_{g\in G_{0}\setminus\left\{ \gamma^{\pm1}\right\} }\int_{a_{g^{-1}}}\left|p_{\gamma}-\eta\right|^{-2\delta}dm_{\delta}\left(\eta\right)=\h\left(p_{\gamma}\right).
\end{eqnarray*}

\end{proof}

Our second technical observation  gives an  asymptotic
estimate for the measure  $\mu_{\delta}$ of
the event $\left\{X_{1}=1, X_{2}=n\right\} $.  In here,  we let $w_{\gamma}$ refer to the width of the cusp associated with $\gamma \in \Gamma_{0}$, that is,  $w_{\gamma}$ is the unique real number such that $\gamma$
can be written as $
\gamma=\tau_{\gamma}^{-1}\circ\ \sigma_{\gamma}\circ\tau_{\gamma}
$, for $\tau_{\gamma}$ and $\sigma_{\gamma}$ given by
$\tau_{\gamma}(\xi):=\xi-p_{\gamma}$ and $\sigma_{\gamma}(\xi):=\xi\left(1+w_{\gamma}\xi\right)^{-1}$.
\begin{lem}
\label{prop:asymptotic-n-excursion}
\begin{align*}
\lim_{n\to \infty} n^{2\delta} \mu_{\delta}\left(\left\{ X_{1}=1,X_{2}=n\right\} \right) = \sum_{\gamma\in\Gamma_{0}}\left(\h\left(p_{\gamma}\right)w_{\gamma}^{-\delta}\right)^{2}.
\end{align*}
\end{lem}
\begin{proof}
First, note that for each $n\in\N$, we have that
\[
\mu_{\delta}\left(\left\{X_{1}=1, X_{2}=n\right\} \right)=\mu_{\delta}\left(\bigcup_{\gamma\in\Gamma_{0}}\;\bigcup_{g_{1},g_{2}\in G_{0}\setminus\left\{ \gamma^{\pm1}\right\}} a_{g_{1}\gamma^{n}g_{2}}\right).
\]
The aim is to show that for each $g_{1}\in G_{0}$,
$\gamma\in\Gamma_{0}\setminus\left\{ g_{1}^{\pm1}\right\} $ and $g_{2}\in G_{0}\setminus\left\{ \gamma^{\pm1}\right\} $,
we have that
\begin{equation}
\lim_{n\to \infty} n^{2\delta} \mu_{\delta}\left(a_{g_{1}\gamma^{n}g_{2}}\right)=\left|g_{1}'\left(p_{\gamma}\right)\right|^{\delta}w_{\gamma}^{-2\delta}\h\left(g_{1}\left(p_{\gamma}\right)\right)\int_{a_{g_{2}}}\left|\xi-p_{\gamma}\right|^{-2\delta}dm_{\delta}\left(\xi\right).\label{eq:tail-asymptotics}
\end{equation}
Before we prove this, let us first see why this is sufficient for the assertion in the lemma. For this, note that  (\ref{eq:tail-asymptotics}) implies that
\begin{eqnarray*}
 &  & \lim_{n\to \infty} n^{2\delta}  \mu_{\delta}\left(\left\{X_{1}=1, X_{2}=n\right\} \right)\\
 & = &\sum_{{\gamma\in\Gamma_{0}}\atop g_{1},g_{2}\in G_{0}\setminus\left\{ \gamma^{\pm1}\right\}}  \left|g_{1}'\left(p_{\gamma}\right)\right|^{\delta}w_{\gamma}^{-2\delta}\h\left(g_{1}\left(p_{\gamma}\right)\right)\int_{a_{g_{2}}}\left|\xi-p_{\gamma}\right|^{-2\delta}dm_{\delta}\left(\xi\right)\\
 & = & \sum_{\gamma\in\Gamma_{0}}w_{\gamma}^{-2\delta}\!\!\sum_{g_{1}\in G_{0}\setminus\left\{ \gamma^{\pm1}\right\} }\!\!\left|g_{1}'\left(p_{\gamma}\right)\right|^{\delta}\h\left(g_{1}\left(p_{\gamma}\right)\right)\sum_{g_{2}\in G_{0}\setminus\left\{ \gamma^{\pm1}\right\} }\int_{a_{g_{2}}}\left|\xi-p_{\gamma}\right|^{-2\delta}dm_{\delta}\left(\xi\right).
\end{eqnarray*}
Using the fact that
\[
\sum_{g_{2}\in G_{0}\setminus\left\{ \gamma^{\pm1}\right\} }\int_{a_{g_{2}}}\left|\xi-p_{\gamma}\right|^{-2\delta}dm_{\delta}\left(\xi\right)=\int_{(a_{\gamma}\,\cup\, a_{\gamma^{-1}})^{{\bf c}}}\left|\xi-p_{\gamma}\right|^{-2\delta}dm_{\delta}\left(\xi\right)=\h\left(p_{\gamma}\right),
\]
it then follows that
\[
\lim_{n\to \infty} n^{2\delta} \mu_{\delta}\left(\left\{X_{1}=1, X_{2}=n\right\} \right)= \sum_{\gamma\in\Gamma_{0}}w_{\gamma}^{-2\delta}\h\left(p_{\gamma}\right)\sum_{g_{1}\in G_{0}\setminus\left\{ \gamma^{\pm1}\right\} }\left|g_{1}'\left(p_{\gamma}\right)\right|^{\delta}\h\left(g_{1}\left(p_{\gamma}\right)\right).
\]
Finally, by combining this observation and Lemma \ref{lem:fixpoint-of-perron-frobenius},
we obtain that
\[
\lim_{n\to \infty} n^{2\delta} \mu_{\delta}\left(\left\{X_{1}=1, X_{2}=n\right\} \right)=\sum_{\gamma\in\Gamma_{0}}\left(\h\left(p_{\gamma}\right)w_{\gamma}^{-\delta}\right)^{2},
\]
which finishes the proof of the lemma.

Therefore, it only remains to verify the asymptotic estimate in (\ref{eq:tail-asymptotics}). For this, let  $g_{1}\in G_{0}$,
$\gamma\in\Gamma_{0}\setminus\left\{ g_{1}^{\pm1}\right\} $ and $g_{2}\in G_{0}\setminus\left\{ \gamma^{\pm1}\right\} $ be fixed.
Using the $\delta$-conformality of $m_{\delta}$ stated in (\ref{eq:delta-conformality}),
it follows that
\begin{eqnarray}
\mu_{\delta}\left(a_{g_{1}\gamma^{n}g_{2}}\right)  =  \int\1_{a_{g_{1}\gamma^{n}g_{2}}}\h\, dm_{\delta}=\int_{a_{g_{2}}}\left|\left(g_{1}\gamma^{n}\right)'\left(\xi\right)\right|^{\delta}\h\left(g_{1}\gamma^{n}\left(\xi\right)\right)dm_{\delta}\left(\xi\right).\label{eq:step1_deltaconformality}
\end{eqnarray}
For $\sigma_{\gamma}$ as given  in the definition of the width of the cusp associated with $\gamma$, stated prior to the lemma,  one immediately verifies that for the $n$-th iterate of $\sigma_{\gamma}$ one has that $\sigma_{\gamma}^{n}\left(\xi\right)=\xi/\left(1+nw_{\gamma}\xi\right)$
and $\left(\sigma_{\gamma}^{n}\right)'\left(\xi\right)=1/\left(1+nw_{\gamma}\xi\right)^{2}$.
Using this, it follows that for all $\xi\in\R \setminus\left\{ p_{\gamma}\right\} $,
we have that
\[
\lim_{n\rightarrow\infty}n^{2}\left(\gamma^{n}\right)'\left(\xi\right)=w_{\gamma}^{-2}\left(\xi-p_{\gamma}\right)^{-2}.
\]
By combining this with the facts that $\lim_{n\to\infty}\sup_{\xi\in a_{g_{2}}}\left\{ \left|\gamma^{n}\left(\xi\right)-p_{\gamma}\right|\right\} =0$
and that $g_{1}'$ is continuous, we conclude that
\begin{equation}
\lim_{n\rightarrow\infty}\sup_{\xi\in a_{g_{2}}}\left\{ n^{2\delta}\left|\left(g_{1}\gamma^{n}\right)'\left(\xi\right)\right|^{\delta}-\left|g_{1}'\left(p_{\gamma}\right)\right|^{\delta}w_{\gamma}^{-2\delta}\left|\xi-p_{\gamma}\right|^{-2\delta}\right\} =0.\label{eq:derivative-uniformconv}
\end{equation}
Similarly, since $\h\circ g_{1}$ is continuous in $p_{\gamma}$,
we have that
\begin{equation}
\lim_{n\rightarrow\infty}\sup_{\xi\in a_{g_{2}}}\left\{ \h\left(g_{1}\gamma^{n}\left(\xi\right)\right)-\h\left(g_{1}\left(p_{\gamma}\right)\right)\right\} =0.\label{eq:h_uniformconv}
\end{equation}
By combining the two observations in (\ref{eq:derivative-uniformconv}) and (\ref{eq:h_uniformconv}),
we obtain that
\[
\lim_{n\rightarrow\infty}\sup_{\xi\in a_{g_{2}}}\left\{ \left|n^{2\delta}\left|\left(g_{1}\gamma^{n}\right)'\left(\xi\right)\right|^{\delta}\h\left(g_{1}\gamma^{n}\left(\xi\right)\right)-\frac{\left|g_{1}'\left(p_{\gamma}\right)\right|^{\delta} \h\left(g_{1}\left(p_{\gamma}\right)\right)}{w_{\gamma}^{2\delta}\left|\xi-p_{\gamma}\right|^{2\delta}}\right|\right\} =0,
\]
which, together with (\ref{eq:step1_deltaconformality}), finishes the
proof of (\ref{eq:tail-asymptotics}).
\end{proof}
\subsection{Inducing and Strong Mixing Properties}
In this section we introduce a certain induced system associated with $\left(L_{r}(G),T,\mu_{\delta}\right)$. Subsequently, we will then show
that this system satisfies  some strong mixing properties, which will be  an important ingredient  in the  proof of
Theorem \ref{thm}. \\
Let us begin by introducing the above-mentioned induced system. For this, we define the set
$\mathcal{D} \subset L_{r}(G)$
by
\[
\mathcal{D}:=\left\{ X_{1}=1\right\}.
\]
Note that the set $\mathcal{D}$ is a Darling-Kac set for $T$ in the sense of \cite[page 123]{MR1450400}. The next step is to introduce
the induced transformation $T_{\mathcal{D}}$ with respect
to the set $\mathcal{D}$. This map is defined, for each $\xi\in\mathcal{D}$, by
\[T_{\mathcal{D}}(\xi):=T^{\rho(\xi)}(\xi),\]
 where $\rho$ denotes the return time
function, given by $\rho(\xi):=\min\{n\in\N:T^{n}(\xi)\in\mathcal{D}\}$.
This then also allows us to introduce the partition $\alpha_{\mathcal{D}}$  induced by $\alpha$ on $\mathcal{D}$, that is,
$\alpha_{\mathcal{D}}$ is given by
\[
\alpha_{\mathcal{D}}:=\bigcup_{n\in\N}\mathcal{D}\cap\left\{ \rho=n\right\} \cap\bigvee_{k=0}^{n+1}T^{-k}\left(\alpha\right).
\]
Moreover, we let $\mu_{\delta,\mathcal{D}}$ refer to the restricted measure, given
by $\mu_{\delta,\mathcal{D}}:=\mu_{\delta}\big|_{\mathcal{D}}$. The induced system is now  given by the triple   $\left(\mathcal{D},T_{\mathcal{D}},\mu_{\delta,\mathcal{D}}\right)$, and our next task will be to establish a certain mixing property for this system. To that end, let us recall
the  following definition of continued fraction mixing from \cite[page 124]{MR1450400}.
\begin{itemize}
\item {\em A dynamical system $\left(\Omega,R, \nu\right)$
is called continued fraction mixing with respect to some partition  $\alpha_{0}$ of the space $\Omega$,
if there exists $l_{1}\in\N$ and a sequence $\left(\epsilon_{n}\right)$ of positive reals which tends
 to zero for $n$ tending to infinity, such
that for each $k \in \N$,  for every $A\in\bigvee_{m=0}^{k-1}R^{-m}\left(\alpha_{0} \right)$
and for all Borel set $B\subset \Omega$, we have  that
\[
\left(1-\epsilon_{n}\right) \nu\left(A\right)\nu \left(B\right) \leq \nu \left(A\cap R^{-\left(k+n\right)}\left(B\right)\right)\le\left(1+\epsilon_{n}\right)\nu \left(A\right)\nu \left(B\right),
\]
where the first inequality has to hold for all $n\in\N$ such that $n\ge l_{1}$, whereas  the second inequality has to hold for all $n\in\N$.}
\end{itemize}

\begin{prop}
\label{prop:inducedsystem-cf-mixing}
The induced system $\left(\mathcal{D},T_{\mathcal{D}},\mu_{\delta,\mathcal{D}}\right)$
is continued fraction mixing with respect to the partition  $\alpha_{\mathcal{D}}$ induced by $\mathcal{D}$.
\end{prop}
\begin{proof}({\sl Sketch})
For the proof, note that in \cite{MR2158407} it was shown that for the induced system $\left(\mathcal{D},T_{\mathcal{D}},\mu_{\delta,\mathcal{D}}\right)$
we have that there exist constants $K>0$ and $c\in(0,1)$
such that for all $n,m \in \N$, and for every $A\in\bigvee_{k=0}^{n-1}T_{\mathcal{D}}^{-k}\left(\alpha_{\mathcal{D}}\right)$
and $B\in\bigvee_{k=0}^{m-1}T_{\mathcal{D}}^{-k}\left(\alpha_{\mathcal{D}}\right)$
with $B\subset T_{\mathcal{D}}^{n}(A)$, we have for $\mu_{\delta,\mathcal{D}}$ almost
all $\eta,\xi\in B$, that
\[
\left|\log\frac{d\mu_{\delta,\mathcal{D}}\circ T_{\mathcal{D},A}^{-n}}{d\mu_{\delta,\mathcal{D}}}(\xi)-\log\frac{d\mu_{\delta,\mathcal{D}}\circ T_{\mathcal{D},A}^{-n}}{d\mu_{\delta,\mathcal{D}}}(\eta)\right|\leq K  c^{m},
\]
where $T_{\mathcal{D},A}^{-n}$ denotes the inverse branch of $T_{\mathcal{D}}^{n}$
mapping $T_{\mathcal{D}}^{n}(A)$ bijectively to $A$.
(Note that the latter property is sometimes also referred to as the  Gibbs-Markov property of $T_{\mathcal{D}}$ with respect to the measure $\mu_{\delta,\mathcal{D}}$).
The next step is to note that this property of $T_{\mathcal{D}}$  allows one to conclude
that there
exist $K_{0}>0$ and $C_{0}\in(0,1)$ such that for all $\varphi\in L^{1}(\mu_{\delta,\mathcal{D}})$
and $n\in\N$, we have (see \cite{MR1450400,MR1107025})
\begin{equation}
\left\Vert {\widehat{T}_{\mathcal{D}}}^{n}\varphi-\mu_{\delta,\mathcal{D}}(\varphi)\right\Vert _{L} \leq K_{0} \, C_{0}^{n} \, \|\varphi\|_{L}.
\label{dual}
\end{equation}
Here, $\|\cdot\|_{L}$ refers to the uniform Lipschitz
norm with respect to $\alpha_{\mathcal{D}}$, which is given by (see \cite[page 541]{MR1107025})
\[
\Vert f\Vert_{L}:=\max_{A\in\alpha_{\mathcal{D}}}\sup_{x\in\Int\left(A\right)}\left|f\left(x\right)\right|+\max_{A\in\alpha_{\mathcal{D}}}\sup_{x,y\in\Int\left(A\right)}\frac{\left|f\left(x\right)-f\left(y\right)\right|}{\left|x-y\right|},
\]
and
$\widehat{T}_{\mathcal{D}}$ denotes the dual operator of $T_{\mathcal{D}}$,
which is given by
\[
\mu_{\delta,\mathcal{D}}(\varphi\cdot ( \psi\circ T_{\mathcal{D}}))=\mu_{\delta,\mathcal{D}}(\widehat{T}_{\mathcal{D}}(\varphi)\cdot\psi),\mbox{ for all }\varphi\in L^{1}(\mu_{\delta,\mathcal{D}}),\psi\in L^{\infty}(\mu_{\delta,\mathcal{D}}).
\]
To finish the proof, we  now employ an argument from  \cite[page 500]{MR1107025}, which guarantees that property (\ref{dual}) of the
dual $\widehat{T}_{\mathcal{D}}$  does indeed imply continued fraction mixing of $T_{\mathcal{D}}$ with
respect to the partition $\alpha_{\mathcal{D}}$.
\end{proof}
In the following, we let  $\mathcal{B}$ denote the Borel $\sigma$-algebra on $\R$. { Also, a set $B\in \mathcal{B}$ is called $\Gamma$-invariant if $B=\gamma (B)$, for all $\gamma\in \Gamma$. Moreover, we  define the normalised measure  $\tilde{\mu}_{\delta,\mathcal{D}}:=\mu_{\delta,\mathcal{D}}/\mu_{\delta,\mathcal{D}}\left(\mathcal{D}\right)$.
\begin{lem} 
\label{trans} For every $\Gamma$-invariant set  $B \in  \mathcal{B}$ the following equivalence holds.
\[ \tilde{\mu}_{\delta,\mathcal{D}} (B) =1 \mbox{ if and only if } m_{\delta} (B) =1.   \]
\end{lem}
\begin{proof}
Clearly, the ``if'' part of the equivalence is trivial. Therefore, it is sufficient to show that $\tilde{\mu}_{\delta,\mathcal{D}} (B) =1$ implies
$m_{\delta} (B) =1$.  For this, note that since $\tilde{\mu}_{\delta,\mathcal{D}} (B^{{\bf c}}) =0$  and $\gamma^{-n}  (B)=B$ for all $\gamma \in \Gamma$, 
 it follows that
$ \sum_{\gamma\in\Gamma}\tilde{\mu}_{\delta,\mathcal{D}}\circ\gamma(B^{{\bf c}}) =0$.
 Furthermore, using  the fact that
 $L_{r}\left(G\right) = \bigcup_{\gamma\in\Gamma} \gamma\left(\mathcal{D}\right)$ and using
  $\delta$-conformality of $m_{\delta}$, we have that
$m_{\delta}=\sum_{\gamma\in\Gamma}m_{\delta}|_{\gamma\left(\mathcal{D}\right)}$ is absolutely continuous with respect to $\sum_{\gamma\in\Gamma}\tilde{\mu}_{\delta,\mathcal{D}}\circ\gamma.$ 
This implies that $m_{\delta} (B^{{\bf c}}) =0$ and hence
finishes the proof of the lemma.
\end{proof}}

We end this section with the following lemma, which  gives the relationship between the process $(X_{k})$
and the return time function $\rho$. This result will turn out to be helpful
in the proof of Theorem \ref{thm}.
\begin{lem}
\label{lem:Xi-vs-returntime}For each $\xi\in\mathcal{D}$ and $n\in\N$, we have that
\[
X_{n+1}\left(\xi\right)=\rho\left(T_{\mathcal{D}}^{n-1}\left(\xi\right)\right).
\]
\end{lem}
\begin{proof}
Let us first show that for each $\xi\in\mathcal{D}$ we have that the assertion of the lemma holds for $n=1$, that is, that we have
$
X_{2}\left(\xi\right)=\rho\left(\xi\right).
$
In order to see this, it is sufficient to consider the following two cases. \\
{\em Case 1:} If  $\xi$ is coded by $\left(g_{1}\gamma^{k}g_{2}\dots\right)$,
 for some $g_{1}\in G_{0}$, $\gamma\in\Gamma_{0}\setminus \{g_{1}^{\pm1}\}$, $g_{2}\in G_{0}\setminus\left\{ \gamma^{\pm1}\right\} $
 and $k\in\N$, then we clearly have that $X_{2}\left(\xi\right)=\rho\left(\xi\right)=k$. This settles the assertion in this case.\\
{\em Case 2:}  If  $\xi$ is coded by $\left(gh...\right)$,
for some $g\in G_{0}$ and $h\in H_{0}\setminus \{g^{\pm1}\}$, then we have that $X_{2}\left(\xi\right)=\rho\left(\xi\right)=1$.
This finishes the proof in this case.

We now proceed by induction as follows.
 Assume that  $X_{k+1}\left(\xi\right)=\rho\left(T_{\mathcal{D}}^{k-1}\left(\xi\right)\right)$,  for each $1\leq k \leq n$ and for some $n \in\N$. By definition of $T_{\mathcal{D}}$, we then have
 \[ T^{n}_{\mathcal{D}} (\xi) = T^{\sum_{k=1}^{n} \rho(T_{\mathcal{D}}^{k-1}(\xi))}(\xi)= T^{\sum_{k=1}^{n} X_{k+1}(\xi)}(\xi).\]
 This shows  that the first entry in the code of $T^{n}_{\mathcal{D}} (\xi) $ is equal to the last symbol of the $(n+1)$-th
 block in the code of $\xi$.
 As in the case $n=1$, we now have to consider two cases, which both lead to
  $X_{n+2}\left(\xi\right)=\rho\left(T_{\mathcal{D}}^{n}\left(\xi\right)\right)$.  This finishes the proof of the lemma.
\end{proof}

\subsection{Extreme Value Theory}

For the proof of Theorem \ref{thm}, we also require the following general
result from Extreme Value Theory. Before stating this result, let us recall the
following distributional mixing conditions for a stationary real-valued
process $(Z_{k})$ with respect to some sequence
$\left(u_{n}\right)$ of  real numbers and with respect to some probability
measure $\P$. Here, the set $J_{n,l}$ is defined for each $n,l\in\N$
by
\begin{align*}
J_{n,l}: & =\big\{\left(\left(\ell_{1},\dots,\ell_{p}\right),\left(\ell_{p+1},\dots,\ell_{p+q}\right)\right):p,q,\ell_{k}\in\N,\mbox{ for }1\le k\le p+q,\\
 & \,\qquad\qquad\qquad\qquad\qquad\qquad\qquad 1\le \ell_{1}<\dots<\ell_{p+q}\le n,\ell_{p+1}-\ell_{p}\ge l\big\},
\end{align*}
and  $\tilde{Y}_{\left(\ell_{1},\dots,\ell_{m}\right)}$ is defined for $m \in\N$
and $\ell_{k}\in\N$, for $ 1 \leq k \leq m$, by
\[
\tilde{Y}_{\left(\ell_{1},\dots,\ell_{m}\right)}:=\max_{k=1,\dots,m}Z_{\ell_{k}}.
\]
Also,   we write $v\ast w$ for the
concatenation of two tupels $v$ and $w$.

\begin{itemize}
\item[]
{\bf Condition $\boldsymbol{D(u_n)}$} ({\cite[(3.2.1)]{MR691492}}). {\em We say that Condition $D\left(u_{n}\right)$
holds if there exists a sequence $\left(l_{n}\right)$ of positive
integers such that $\lim_{n}l_{n}/n=0$ and}
\[
\lim_{n\rightarrow\infty}\sup_{\left(v,w\right)\in J_{n,l_{n}}}\left|\P\left(\left\{ \tilde{Y}_{v\ast w}\le u_{n}\right\} \right)-\P\left(\left\{ \tilde{Y}_{v}\le u_{n}\right\} \right)\P\left(\left\{ \tilde{Y}_{w}\le u_{n}\right\} \right)\right|=0.
\]
\item[]
{\bf Condition $\boldsymbol{D'(u_n)}$} (\cite[(3.4.3)]{MR691492}). {\em We say that Condition $D'\left(u_{n}\right)$
holds if we have that }
\[
\lim_{k\rightarrow\infty}\limsup_{n\rightarrow\infty}\, n\sum_{j=2}^{\left\lfloor n/k\right\rfloor }\P\left(\left\{ Z_{1}>u_{n},Z_{j}>u_{n}\right\} \right)=0.
\]
\end{itemize}
\begin{prop}
{\emph{(\cite[Theorem 3.4.1]{MR691492})} \label{prop:Leadbetter}Let
$\left(u_{n}\right)$ be a sequence such that $D\left(u_{n}\right)$
and $D'\left(u_{n}\right)$ hold for the stationary real-valued process
$(Z_{k})$ with probability measure $\P$, and
let $0\le\tau\le\infty$. Then the following equivalence holds.
\begin{equation}
\lim_{n\rightarrow\infty}\P\left(\left\{ \max_{k=1,\dots,n}Z_{k}\le u_{n}\right\} \right)=\e^{-\tau}\,\,\mbox{if and only if }\,\,\lim_{n\rightarrow\infty} n \, \P\left(\left\{ Z_{1}>  u_{n}\right\} \right)=\tau.\label{eq:leadbetter-thm}
\end{equation}
}
\end{prop}

\section{Proofs of the  Main Results }
\begin{proof}
[Proof of Theorem \ref{thm}]The proof of Theorem \ref{thm}  consists of verifying
the following three claims.
\begin{description}
\item [{Claim~1.}] The process $(X_{k+1}^{2\delta-1})$
is stationary with respect to the probability measure {$\tilde{\mu}_{\delta,\mathcal{D}}$}
absolutely continuous with respect to $m_{\delta}$. Further,
Condition $D\left(v_{n}\left(s\right)\right)$ and Condition $D'\left(v_{n}\left(s\right)\right)$
are satisfied with respect to the sequence $(v_{n}(s))$, given by $v_{n}\left(s\right):=\left(n+1\right)s$,
for all $n \in \N$ and $s>0$.
\item [{Claim~2.}] The right hand side of (\ref{eq:leadbetter-thm}) in
Proposition \ref{prop:Leadbetter} is satisfied for the stationary
process $(X_{k+1}^{2\delta-1})$ with respect
to the sequence $(v_{n}\left(s\right))$ given in Claim 1.
\item [{Claim~3.}] If the first assertion of the equivalence in (\ref{eq:leadbetter-thm}) in
Proposition \ref{prop:Leadbetter} holds for some probability measure
absolutely continuous with respect to $m_{\delta}$, then it holds
for all probability measures absolutely continuous with respect to
$m_{\delta}$.
\end{description}
\noindent \textit{Claim 1}: In order to show that $(X_{k+1}^{2\delta-1})$
is a stationary process, note that it is well-known that the $T$-invariance
of $\mu_{\delta}$ implies
that $\tilde{\mu}_{\delta,\mathcal{D}}$ is $T_{\mathcal{D}}$-invariant
(see \cite[Theorem 17.1.3]{MR1419320}). Since by Lemma \ref{lem:Xi-vs-returntime},
we have for all $k\in\N$ and $\xi\in\mathcal{D}$ that
$
X_{k+1}\left(\xi\right)=\rho\left(T_{\mathcal{D}}^{k-1}(\xi)\right),
$
it follows that $(X_{k+1}^{2\delta-1})$ is a
stationary process.  Next, we verify that Condition $D\left(v_{n}\left(s\right)\right)$
and Condition $D'\left(v_{n}\left(s\right)\right)$ hold, for the sequence $(v_{n}(s))$ defined in Claim 1. Since the dynamical system $\left(\mathcal{D},T_{\mathcal{D}},\tilde{\mu}_{\delta,\mathcal{D}}\right)$
is continued fraction mixing with respect to the induced partition
$\alpha_{\mathcal{D}}$,
there exists a sequence $\left(l_{n}\right)$ with $\lim_{n}l_{n}/n=0$
such that Condition $D\left(v_{n}\left(s\right)\right)$ holds
for the process ($X_{k+1}^{2\delta-1})$ with
respect to $\tilde{\mu}_{\delta,\mathcal{D}}$.
The next aim is to show that  Condition $D'\left(v_{n}\left(s\right)\right)$ also
holds for this process, that is, that we have
\begin{equation}
\lim_{k\rightarrow\infty}\limsup_{n\rightarrow\infty}\, n\sum_{j=2}^{\left\lfloor n/k\right\rfloor }\tilde{\mu}_{\delta,\mathcal{D}}\left(\left\{ X_{2}^{2\delta-1}>
v_{n}(s),X_{j+1}^{2\delta-1}>v_{n}(s) \right\} \right)=0.\label{eq:condition_D'}
\end{equation}
Using the fact that  $\left(\mathcal{D},T_{\mathcal{D}},\tilde{\mu}_{\delta,\mathcal{D}}\right)$
is continued fraction mixing with respect to $\alpha_{\mathcal{D}}$,
 it follows that there exists a constant $C>0$ such that
\begin{eqnarray*}
  \tilde{\mu}_{\delta,\mathcal{D}} && \hspace{-9mm} \left(\left\{ X_{2}^{2\delta-1}>v_{n}(s),X_{j+1}^{2\delta-1}>v_{n}(s) \right\} \right)\\
 & \le & \left(1+C\right)  \tilde{\mu}_{\delta,\mathcal{D}}\left(\left\{ X_{2}^{2\delta-1}>v_{n}(s) \right\} \right)\tilde{\mu}_{\delta,\mathcal{D}}\left(\left\{ X_{j+1}^{2\delta-1}>v_{n}(s) \right\} \right).
\end{eqnarray*}
Combining this with the stationarity of $(X_{k+1}^{2\delta-1})$,
we obtain that
\begin{eqnarray*} & &   \sum_{j=2}^{\left\lfloor n/k\right\rfloor }\tilde{\mu}_{\delta,\mathcal{D}}\left(\left\{ X_{2}^{2\delta-1}>v_{n}(s),X_{j+1}^{2\delta-1}>v_{n}(s) \right\} \right)\\
& \le & \left(1+C\right)  \, \sum_{j=2}^{\left\lfloor n/k\right\rfloor }\tilde{\mu}_{\delta,\mathcal{D}}\left(\left\{ X_{2}^{2\delta-1}>v_{n}(s) \right\} \right)
 \tilde{\mu}_{\delta,\mathcal{D}}\left(\left\{ X_{2}^{2\delta-1}>v_{n}(s) \right\} \right)\\
 & \leq & \left(1+C\right) n/k \, \left(\tilde{\mu}_{\delta,\mathcal{D}}\left(\left\{ X_{2}^{2\delta-1}>v_{n}(s) \right\} \right)\right)^{2}.
\end{eqnarray*}
Using (\ref{eq:tail-asymptotic}), which will be obtained in Claim
2 below, we have that
\[
\limsup_{n\rightarrow\infty}n^{2}/k \, \, \tilde{\mu}_{\delta,\mathcal{D}}\left(\left\{ X_{2}^{2\delta-1}>v_{n}(s) \right\} \right)^{2}= k^{-1}\left(s^{-1}\kappa\left(G\right)\right)^{2},
\]
which tends to zero for $k$ tending to infinity. This completes the
proof of (\ref{eq:condition_D'}).

\noindent \textit{Claim 2}: The aim is to show that for each $s>0$, we have that
\begin{equation}
\lim_{n\to \infty} \left( (n+1)\tilde{\mu}_{\delta,\mathcal{D}}\left(\left\{ X_{2}^{2\delta-1}>  v_{n}\left(s\right)\right\} \right)\right)= s^{-1}\kappa\left(G\right).\label{eq:tail-asymptotic}
\end{equation}
In order to prove this, let us  first remark that
\[
\tilde{\mu}_{\delta,\mathcal{D}}\left(\left\{ X_{2} > v_{n}(s)^{1/(2\delta-1)} \right\} \right)\\
  =  \sum_{k\ge\left\lceil v_{n}(s)^{1/\left(2\delta-1\right)}\right\rceil }\tilde{\mu}_{\delta,\mathcal{D}}\left(\left\{ X_{1}=k\right\} \right).
\]
By Lemma \ref{prop:asymptotic-n-excursion}, we have, for $n$ tending to infinity, that the asymptotic behaviour of
\[  \sum_{k\ge\left\lceil v_{n}(s)^{1/\left(2\delta-1\right)}\right\rceil }  \tilde{\mu}_{\delta,\mathcal{D}}\left(\left\{ X_{2}=k\right\} \right)\]
coincides with the asymptotic behaviour of
\[  \sum_{k\ge\left\lceil v_{n}(s)^{1/\left(2\delta-1\right)}\right\rceil } \mu_{\delta}\left(\mathcal{D}\right)^{-1} k^{-2\delta}\sum_{\gamma\in\Gamma_{0}} \left(\h\left(p_{\gamma}\right) w_{\gamma}^{-\delta}\right)^{2},\]
where, by definition of $\kappa\left(G\right)$, the latter sum is equal to
 \[ \sum_{k\ge\left\lceil v_{n}(s)^{1/\left(2\delta-1\right)}\right\rceil } k^{-2\delta} \kappa\left(G\right)\left(2\delta-1\right).
\]
Now, note that one immediately verifies, by using the integral test, that
\[
\lim_{n\to \infty}  \sum_{k\ge\left\lceil v_{n}(s)^{1/\left(2\delta-1\right)}\right\rceil } (n+1) k^{-2\delta}  =   \left((2\delta-1)s \right)^{-1}.
\]
By inserting this into the calculation above, the assertion in   (\ref{eq:tail-asymptotic}) follows.
We can now combine (\ref{eq:tail-asymptotic}) and Proposition \ref{prop:Leadbetter}, which gives  that
for each $s>0$, we have
\[
\lim_{n\to\infty}\tilde{\mu}_{\delta,\mathcal{D}}\left(\left\{ \max_{k=2,\dots,n+1}X_{k}^{2\delta-1}/(n+1)\le s\right\} \right)=\exp\left(-\kappa\left(G\right)/s\right).
\]
Finally, recall that $\tilde{\mu}_{\delta,\mathcal{D}}$ is supported on  $\left\{ X_{1}=1\right\}$, and hence, if  we include $X_{1}$ in the preceding expression, then this does not alter the maximum. Therefore,
it follows that

\noindent
\[
\lim_{n\to\infty}\tilde{\mu}_{\delta,\mathcal{D}}\left(\left\{ Y_{n}{}^{2\delta-1}/n\le s\right\} \right)=\exp\left(-\kappa\left(G\right)/s \right).
\]

\noindent \textit{Claim 3}: The aim is to show that for each probability
measure $\nu$ absolutely continuous with respect to $m_{\delta}$,
we have that the following distributional convergence holds:
\begin{equation}
\frac{1}{n}\left(Y_{n}^{2\delta-1}-Y_{n}^{2\delta-1}\circ T\right)\overset{\nu}{\longrightarrow}0,\mbox{ for }n\mbox{ tending to infinity}.\label{eq:thaler-condition-for-strongdistributionalconv}
\end{equation}
This will be sufficient for the proof of Claim 3, since by combining (\ref{eq:thaler-condition-for-strongdistributionalconv}) with
 Aaronson's Compactness Theorem (see \cite[Proposition 0]{MR632462},
 \cite{MR1603998}), we have that if $(Y_{n}^{2\delta-1}/n)$ converges in distribution with
respect to some particular probability measure absolutely continuous
with respect to $m_{\delta}$, then  it
converges in distribution with respect to {\em every} probability measure
absolutely continuous with respect to $m_{\delta}$.

For the proof of (\ref{eq:thaler-condition-for-strongdistributionalconv}),
let us fix a probability measure $\nu$ absolutely continuous with respect to $m_{\delta}$. For each $\epsilon>0$
and for all $n>1/\epsilon$, we then have that
\begin{eqnarray*}
\left\{ \left|Y_{n}^{2\delta-1}-Y_{n}^{2\delta-1}\circ T\right|>\epsilon n\right\}  & \subset & \mathcal{D}\cap\left\{ \left|\max_{k=1,\dots,n}X_{k}^{2\delta-1}-\max_{k=2,\dots,n+1}X_{k}^{2\delta-1}\right|>\epsilon n\right\} \\
 & \subset & \mathcal{D}\cap\left\{ X_{1}^{2\delta-1}>\epsilon n\,\vee\, X_{n+1}^{2\delta-1}>\epsilon n\right\} \\
 & = & \mathcal{D}\cap\left\{ X_{n+1}^{2\delta-1}>\epsilon n\right\} .
\end{eqnarray*}
Hence, it follows that
\[
\nu\left(\left\{ \left|Y_{n}^{2\delta-1}-Y_{n}^{2\delta-1}\circ T\right|>\epsilon n\right\} \right)\le\nu\left(\mathcal{D}\cap\left\{ X_{n+1}^{2\delta-1}>\epsilon n\right\} \right).
\]
In order to finish the proof, it remains to show that $\nu(\mathcal{D}\cap\{ X_{n+1}^{2\delta-1}>\epsilon n\} )$ tends to zero,
for $n$ tending to infinity. In order to see this, note that for $\varphi:=d\nu/d\mu_{\delta}\in L^{1}\left(\mu_{\delta}\right)$ we have, using Lemma \ref{lem:Xi-vs-returntime},  that
\begin{eqnarray*}
\nu\left(\mathcal{D}\cap\left\{ X_{n+1}^{2\delta-1}>\epsilon n\right\} \right) & = & \int\1_{\mathcal{D}\cap\left\{ X_{n+1}^{2\delta-1}>\epsilon n\right\} }\cdot \varphi\; d\mu_{\delta}\\
 & = & \int\1_{\left\{ X_{n+1}^{2\delta-1}>\epsilon n\right\} }\cdot \varphi\, \;d\mu_{\delta,\mathcal{D}}\\
 & = & \int\left(\1_{\left\{ X_{2}^{2\delta-1}>\epsilon n\right\} }\circ T_{\mathcal{D}}^{n-1}\right)\cdot \varphi\, \;d\mu_{\delta,\mathcal{D}}.
\end{eqnarray*}
Since the set of Lipschitz continuous functions is dense in $L^{1}\left(\mu_{\delta,\mathcal{D}}\right)$, we have that for each $\epsilon'>0$ there exists
 a Lipschitz continuous function $\psi$ such that $\Vert\psi-\varphi_{|\mathcal{D}}\Vert_{1}<\epsilon'$.
By inserting this into the above calculation,
we obtain that
\begin{eqnarray*}
\nu\left(\mathcal{D}\cap\left\{ X_{n+1}^{2\delta-1}>\epsilon n\right\} \right) & \le & \epsilon'+\int\left(\1_{\left\{ X_{2}^{2\delta-1}>\epsilon n\right\} }\circ T_{\mathcal{D}}^{n-1}\right)\cdot \psi\, \;d\mu_{\delta,\mathcal{D}}\\
 & = & \epsilon'+\int\1_{\left\{ X_{2}^{2\delta-1}>\epsilon n\right\} }\cdot  \hat{T}_{\mathcal{D}}^{n-1}\left(\psi\right)\, \;d\mu_{\delta,\mathcal{D}}.
\end{eqnarray*}
For the second summand in the latter expression, we then observe that, by (\ref{dual}) in the proof of Proposition \ref{prop:inducedsystem-cf-mixing}, we have
that
\begin{eqnarray*}
 \limsup_{n\rightarrow\infty}\int\1_{\left\{ X_{2}^{2\delta-1}>\epsilon n\right\} }\cdot \hat{T}_{\mathcal{D}}^{n-1}\left(\psi\right)\, \;d\mu_{\delta,\mathcal{D}} && \\ &&  \!\!\!\!\!\!\!\!\!\!\!\!\!\!\!\!\!\!\!\!\!\! \le\limsup_{n\rightarrow\infty}\mu_{\delta,\mathcal{D}}\left(\psi\right)\int\1_{\left\{ X_{2}^{2\delta-1}>\epsilon n-1\right\} }\, \;d\mu_{\delta,\mathcal{D}}=0.
\end{eqnarray*}
Note that in here, the second expression vanishes, since  $\left(\{ X_{2}^{2\delta-1}>\epsilon n\} \right)$ is a nested sequence of sets which
decreases to the empty set. Therefore, since $\epsilon'>0$ was chosen to be arbitrary,
the statement in (\ref{eq:thaler-condition-for-strongdistributionalconv})
follows. This finishes the proof of Claim 3 and hence, the proof of
Theorem \ref{thm} is complete.
\end{proof}
{
\begin{proof}
[Proof of Theorem $\ref{prop:loglogliminf}$] First, note that the sets considered in both assertions of Theorem \ref{prop:loglogliminf} are 
  $\Gamma$-invariant. Hence, using Lemma \ref{trans}, it is sufficient to prove the assertions for $\tilde{\mu}_{\delta,\mathcal{D}}$ instead of $m_{\delta}$. The proof of the
first assertion in the theorem is a straight-forward adaptation of the proof of
\cite[Theorem 1]{MR0387226}. 
The first necessary modification of the proof  in  \cite{MR0387226} is to  alter the definition of the process $L\left(M,N\right)$ considered in \cite{MR0387226}.  Namely, instead of  $L\left(M,N\right)$,
here we have to consider the process
\[
\tilde{L}\left(M,N\right):=\max_{n=M,\dots,M+N}Y_{n}^{2\delta-1},\mbox{ for }M,N\in\N.
\]
The second modification is that  instead of the normalising sequence $(\psi\left(n\right))$, introduced by Philipp in \cite{MR0387226},
we have to consider the normalising sequence $(\tilde{\psi} \left(n\right))$, given by
\[
\tilde{\psi} \left(n\right):= \kappa\left(G\right) n/\log\log n,\mbox{ for }n\in\N.
\]
Also, note that the type of mixing for the Gauss system, which is established  in \cite[Lemma 2]{MR0387226} and which is vital in the proof of \cite[Theorem 1]{MR0387226}, certainly also holds
for  our induced system $\left(\mathcal{D},T_{\mathcal{D}},\tilde{\mu}_{\delta,\mathcal{D}}\right)$.
This is an immediate consequence of Proposition \ref{prop:inducedsystem-cf-mixing}, which guarantees that $\left(\mathcal{D},T_{\mathcal{D}},\tilde{\mu}_{\delta,\mathcal{D}}\right)$ is continued fraction mixing with respect to the partition
$\alpha_{\mathcal{D}}$.  
Let us now  consider the sets $E_k:= \left\{ \tilde{L}\left(k^{2k},k^{2(k+1)}\right) \le \tilde{\psi}(k^{2(k+1)}) \right\}$. Since $\tilde{\mu}_{\delta,\mathcal{D}}$  is $T_{\mathcal{D}}$-invariant, we have  for each $k\in\N$  that 
\[
\tilde {\mu}_{\delta,\mathcal{D}} (E_k)
= \tilde {\mu}_{\delta,\mathcal{D}} \left( \left\{ \tilde{L}\left(0,k^{2(k+1)}\right)\le 
\tilde{\psi}(k^{2(k+1)})   \right\} \right).
\]
Using Theorem $\ref{thm}$, it follows that there is a constant $C_1>0$ such that for each $k>1$ we have
\[ 
\tilde {\mu}_{\delta,\mathcal{D}} (E_k) \ge C_1 \exp(-\log\log(k^{2(k+1)}))\ge C_1(4k\log{k})^{-1}.
\]
 Furthermore,  note that  $E_k\in \bigvee_{j=k^{2k}}^{k^{2(k+1)}+k^{2k}}T_{\mathcal{D}}^{-j}(\alpha_{\mathcal{D}})$ and  that $\lim_{k\rightarrow \infty}(k+1)^{2(k+1)}-k^{(2(k+1)}-k^{2k}=\infty$.  Since the induced system $\left(\mathcal{D},T_{\mathcal{D}},\tilde{\mu}_{\delta,\mathcal{D}}\right)$ is continued fraction mixing with respect to $\alpha_{\mathcal{D}}$, we can now apply  the second Borel-Cantelli Lemma (see for instance \cite{sprin}), which gives that  $\tilde {\mu}_{\delta,\mathcal{D}}$  almost every $\xi \in \mathcal{D}$ is contained in infinitely many of the sets $E_k$.  Similarly, let us now consider the sets   $F_k:= \left\{ \tilde{L}\left(0,k^{2k}\right) \ge \tilde{\psi}(k^{2(k+1)}) \right\}$. Note that that there exist constants $C_{2},C_{3}>0$ such that for each $k\in\N$ we have 
\begin{eqnarray*}
\tilde {\mu}_{\delta,\mathcal{D}} (F_k)
&\le& \sum_{j=1}^{ k^{2k} }\tilde {\mu}_{\delta,\mathcal{D}} \left( \left\{  X_j^{2\delta-1}\ge \tilde{\psi}(k^{2(k+1)})   \right\} \right)\\ &=&
k^{2k}\tilde {\mu}_{\delta,\mathcal{D}} \left( \left\{  X_2^{2\delta-1}\ge \tilde{\psi}(k^{2(k+1)})   \right\} \right)\\
& \leq & C_{2} k^{2k} \left( \tilde{\psi}(k^{2(k+1)})\right)^{-1} \leq C_{3} k^{-3/2}.
\end{eqnarray*}
Using the Borel-Cantelli Lemma, it follows that  $\tilde {\mu}_{\delta,\mathcal{D}}$  almost every $\xi \in \mathcal{D}$ is contained in at most finitely many of the sets $F_k$. Combining these two  observations and noting that 
$E_k \setminus F_k =   \left\{ \tilde{L}\left(0,k^{2k}+k^{2(k+1)}\right) \le \tilde{\psi}(k^{2(k+1)}) \right\}$ is a subset of  $\left\{ \tilde{L}\left(0,k^{2(k+1)}\right) \le \tilde{\psi}(k^{2(k+1)}) \right\}$, we conclude that $\tilde {\mu}_{\delta,\mathcal{D}}$  almost every $\xi \in \mathcal{D}$  is contained in  infinitely many sets of the form  $\left\{ \tilde{L}\left(0,k^{2(k+1)}\right) \le \tilde{\psi}(k^{2(k+1)}) \right\}$.
This implies that $\tilde {\mu}_{\delta,\mathcal{D}}$  almost everywhere we have that
\[
\liminf_{n\rightarrow\infty}\, Y_{n}^{2\delta-1} (\log\log n)/n\le\kappa\left(G\right).
\]
In order to prove that in here the reverse inequality also holds,  consider the sets $G_k:= \left\{ \tilde{L}\left(0,\lfloor r^k \rfloor \right) \le r^{-2} \tilde{\psi}( \lfloor  r^{k+1}  \rfloor ) \right\}$, for some fixed $r>1$. Since the convergence in Theorem $\ref{thm}$  is uniform in $\R$, it follows that there exists a constant $C_4>0$ such that for each $k\in\N$ we have 
\[
\tilde {\mu}_{\delta,\mathcal{D}} (G_k)
\le C_4 \exp(-r\log\log{r^k})\le C_4 k^{-r}.
\]
By employing the Borel-Cantelli Lemma once more,  it follows that for $\tilde {\mu}_{\delta,\mathcal{D}}$  almost every  $\xi \in \mathcal{D}$  there exists $k_0 \in \N$  such that  $\tilde{L}\left(0,\lfloor r^k \rfloor \right) ( \xi) > r^{-2}\tilde{\psi}(\lfloor r^{k+1}\rfloor) $, for all $k\ge k_0$. Clearly, for each $N>\lfloor r^{k_0} \rfloor$  there is  $k\ge k_0$ such that $\lfloor r^k \rfloor \le N < r^{k+1}$.  Using this, it follows that  $\tilde{L}\left(0,\lfloor r^k \rfloor \right) (\xi)  \le  Y_{N}^{2\delta-1}(\xi)$ and that $\tilde{\psi}(N)\le\tilde{\psi}(\lfloor r^{k+1}\rfloor)$. This shows that $ Y_{N}^{2\delta-1}(\xi) > r^{-2}\tilde{\psi}(N)$, for each $N>\lfloor r^{k_0} \rfloor$. Therefore, by letting $r$ tend to $1$ from above, we can now conclude that  $\tilde {\mu}_{\delta,\mathcal{D}}$  almost everywhere we have 
\[
\liminf_{n\rightarrow\infty}\, Y_{n}^{2\delta-1} (\log\log n)/n\ge\kappa\left(G\right).
\]
This finishes the proof of the first assertion in Theorem $\ref{prop:loglogliminf}$.

For the proof of the second assertion in the theorem, note that Proposition \ref{prop:measure}  immediately implies that there exists a constant $C_{5}>0$
such that
\[ C_{5}^{-1} \ell_{n}^{-1} \leq m_{\delta} \left(\left\{X_{n}^{2 \delta -1} \geq \ell_{n}\right\}\right) \leq  C_{5} \ell_{n}^{-1}. \]
Therefore, since  $\tilde{\mu}_{\delta, \mathcal{D}}$ and $m_{\delta}$ are comparable on $\{X_{1}=1\}$ and since  $\left(\mathcal{D},T_{\mathcal{D}},\tilde{\mu}_{\delta,\mathcal{D}}\right)$
is continued fraction mixing with respect to the partition  $\alpha_{\mathcal{D}}$, we can once more apply the Borel-Cantelli Lemma
and the second Borel-Cantelli Lemma, which gives
\[
\tilde{\mu}_{\delta, \mathcal{D}} \left(\left\{ X_{n}^{2 \delta -1}\geq \ell_{n} \mbox{ for infinitely many $n$} \right\}\right) =
\left\{ \begin{array}{ccc}
0 & \text{ if } & \sum_{n \in \N} \ell_{n}^{-1} \mbox{ converges} \\
1 & \text{ if } &  \sum_{n \in \N} \ell_{n}^{-1} \mbox{ diverges}. \end{array}\right.\]
We then proceed as follows. On the one hand, if  $\sum_{n \in \N} \ell_{n}^{-1}$ converges, then choose  a monotone increasing sequence $(d_{n})$
which tends to infinity, for $n$ tending to infinity, such that $  \sum_{n \in \N} d_{n }\ell_{n}^{-1} $ still converges. In this situation, we then have
\[  \tilde{\mu}_{\delta,\mathcal{D}} \left(\left\{ X_{n}^{2 \delta -1}/\ell_{n} <  1/d_{n} \mbox{ for at most finitely many $n$}   \right\}\right) =1 ,\]
which implies that $\tilde{\mu}_{\delta,\mathcal{D}}$ almost everywhere we have 
\[ \limsup_{n \to \infty}  X_{n}^{2 \delta -1}/\ell_{n} =0.\]
On the other hand, if $\sum_{n \in \N} \ell_{n}^{-1}$ diverges, then choose  a monotone decreasing sequence $(e_{n})$ which tends to zero, for $n$ tending to infinity,  such that  we still have that $ \sum_{n \in \N} e_{n }\ell_{n}^{-1} $ diverges. It then follows that
\[  \tilde{\mu}_{\delta,\mathcal{D}} \left(\left\{X_{n}^{2 \delta -1}/\ell_{n}\geq 1/e_{n} \mbox{ for infinitely many $n$}   \right\}\right) =1 .\]
This implies that  $\tilde{\mu}_{\delta,\mathcal{D}}$ almost everywhere, we have 
\[ \limsup_{n \to \infty} X_{n}^{2 \delta -1} / \ell_{n} =\infty.\]
By combining these two observations and then using Lemma \ref{trans}, it now follows that  $m_{\delta}$ almost everywhere we have that
\[ \limsup_{n \to \infty} X_{n}^{2 \delta -1} / \ell_{n}\in \{0,\infty\}.\]
Now, the proof of the second assertion in Theorem \ref{prop:loglogliminf} follows from the following elementary observation.
\begin{itemize}
\item[] {\bf Limsup-Max Principle.}  {\em Let $\left(p_{n}\right)$ and $\left(q_{n}\right)$ be two arbitrary sequences
of positive real numbers  such that
 $\left(q_{n}\right)$ is unbounded and non-decreasing.
We then have that }
\begin{equation}
\limsup_{n\to\infty}\frac{p_{n}}{q_{n}}=\limsup_{n\rightarrow\infty}\frac{\max_{k=1,\ldots,n}p_{k}}{q_{n}}.\label{eq:maximum-exercise}
\end{equation}
\end{itemize}
  \end{proof}
}
\begin{proof}
[Proof of Corollary $\ref{cor:log-khintchine}$] For the upper bound
of the first assertion of the corollary, observe that by Proposition
\ref{prop:measure}, we have that for each  $\epsilon>0$ there exists a constant $C_{\epsilon}>0$ such that for each  $n \in \N$, we have that
\[
C_{\epsilon}^{-1} n^{-\left(1+\epsilon\right)} \leq m_{\delta}\left(\left\{ \frac{\log X_{n}}{\log n}\ge\frac{1+\epsilon}{2\delta-1}\right\} \right)\leq C_{\epsilon} n^{-\left(1+\epsilon\right)}.
\]
By applying the Borel-Cantelli Lemma and then letting $\epsilon$ tend
to zero, we obtain that $m_{\delta}$ almost everywhere we have that
\begin{equation}
\limsup_{n\rightarrow\infty}\frac{\log X_{n}}{\log n}\le\frac{1}{2\delta-1}.\label{eq:upperbound}
\end{equation}
By employing the Limsup-Max Principle (\ref{eq:maximum-exercise}), it follows that $m_{\delta}$ almost everywhere we have that
\[
\limsup_{n\to \infty}\frac{\log Y_{n}}{\log n}\le\frac{1}{2\delta-1}.
\]
Clearly, the lower bound
of the first assertion of the corollary is an immediate consequence of Theorem \ref{prop:loglogliminf}.
 Therefore, we have now shown  that $m_{\delta}$ almost everywhere we have that
\[
\lim_{n\rightarrow\infty}\frac{\log Y_{n}}{\log n}=\frac{1}{2\delta-1},
\]
which finishes the proof of the first part of the corollary.

Finally, let us show how to derive the second assertion of the corollary
from the first. For this, recall that each  $\xi\in\mathcal{D}$  can be coded
 by $(g_{1}^{X_{1}(\xi)} g_{2}^{X_{2}(\xi)}...)$, for some uniquely determined $g_{1},g_{2},... \in G_{0}$.
We then
 define the cocycle $I^{*}:\mathcal{D} \to \R$ and its  Birkhoff sum $t_{n}$ with respect to $T_{\mathcal{D}}$, for all $\xi\in\mathcal{D}$,  by (see \cite[3.1.3]{MR2041265})
\[
I^{*}\left(\xi\right):=-\log\left| \left(g_{1}^{X_{1}(\xi)}\right)' \left(g_{1}^{-X_{1}(\xi)}\right)\right|\,\,\mbox{ and   }  \, \, t_{n}\left(\xi\right):=\sum_{k=0}^{n-1}I^{*}\circ T_{\mathcal{D}}^{k}\left(\xi\right).
\]
Since $\left(\mathcal{D},T_{\mathcal{D}},\tilde{\mu}_{\delta,\mathcal{D}}\right)$
is ergodic, we have  $\tilde{\mu}_{\delta,\mathcal{D}}$ almost
everywhere and  for some constant $C>0$ that (see \cite[(3.1.3)]{MR2041265})
\[
\lim_{n \to \infty}\frac{t_{n}}{n}= \int I^{*}\;d\tilde{\mu}_{\delta,\mathcal{D}}\leq C \sum_{k=1}^{\infty}k^{-2\delta}\log k<\infty.
\]
This  implies that $\lim_{n} \log t_{n} (\xi) / \log n =1$, for $\tilde{\mu}_{\delta,\mathcal{D}}$ almost every $\xi\in\mathcal{D}$. Moreover, since $\lim_{n} \log (n+1)  / \log n =1$, it follows that
$\lim_{n} \log t / \log n =1$, for  all
$t \in [t_{n}\left(\xi\right), t_{n+1} \left(\xi\right)]$.
Next, note that by elementary hyperbolic geometry (see for instance \cite{MR568933,MR1346819}),
we have that there exists a constant $K>0$ such that for all $\xi \in \mathcal{D}$ and $n\in\N$ we have
\[
K^{-1} \, X_{n}\left(\xi\right)\leq \exp\left(d\left(\xi_{\left(t_{n}\left(\xi\right)+t_{n+1}\left(\xi\right)\right)/2},G\left(i\right)\right)\right) \leq K \, X_{n}\left(\xi\right).
\]
Combining these observations, it follows that for $\tilde{\mu}_{\delta,\mathcal{D}}$ almost every $\xi\in\mathcal{D}$ we have 

\[
\lim_{T \to \infty} \, \, \max_{0\le t\le T}\frac{d\left(\xi_{t},G\left(i\right)\right) }{\log T}  =  \lim_{n\to \infty} \, \, \max_{k=1,\dots,n}\frac{d\left(\xi_{\left(t_{k}\left(\xi\right)+t_{k+1}\left(\xi\right)\right)/2},G\left(i\right)\right)}{\log n}
  =  \lim_{n\to \infty} \frac{\log Y_{n}\left(\xi\right)}{\log n}.
\]
Now, an application of Lemma \ref{trans} finishes the proof of the corollary.
\end{proof}

\providecommand{\bysame}{\leavevmode\hbox to3em{\hrulefill}\thinspace}
\providecommand{\MR}{\relax\ifhmode\unskip\space\fi MR }
\providecommand{\MRhref}[2]{%
  \href{http://www.ams.org/mathscinet-getitem?mr=#1}{#2}
}
\providecommand{\href}[2]{#2}

\end{document}